\documentclass{amsart}
\usepackage{color}
\usepackage{amssymb,latexsym,amsmath, amsfonts}

\newcommand{\N}{{\mathbb N}}
\newcommand{\Z}{{\mathbb Z}}
\newcommand{\R}{{\mathbb R}}
\newcommand{\C}{{\mathbb C}}
\newcommand{\cM}{{\mathcal M}}
\newcommand{\cF}{{\mathcal F}}
\newcommand{\cL}{{\mathcal L}}
\newcommand{\cD}{{\mathcal D}}

\newcommand{\eli}{\ell_\infty(\mathbb{Z})}

\newcommand{\smt}{ S(\cM,\tau)}

\DeclareMathOperator{\tr}{tr}
\DeclareMathOperator{\rank}{rank}

 \newtheorem{thm}{Theorem}[section]
\newtheorem{cor}[thm]{Corollary}
\newtheorem{lem}[thm]{Lemma}

\theoremstyle{definition}
\newtheorem{defn}[thm]{Definition}
\theoremstyle{remark}
\newtheorem{rem}[thm]{Remark}

\numberwithin{equation}{section}

\title[Pietsch correspondence]{Pietsch correspondence for symmetric functionals on Calkin operator spaces associated with  semifinite von Neumann algebras}
\author[Galina Levitina {\protect \and} Alexandr Usachev]{Galina Levitina {\protect
\and} Alexandr Usachev}
\address{Galina Levitina, School of Mathematics and Statistics, The University of New South Wales, 2052, NSW, Australia.}
\email{g.levitina@unsw.edu.au}
\address{Alexandr Usachev, Central South University, Hunan, China.}
\email{a.usachev.@csu.edu.cn}

\date{\today}

\begin{document}
\begin{abstract}
In this paper we extend the Pietsch correspondence for ideals of compact operators and traces on them to the semifinite setting. We prove that a shift-monotone space $E(\Z)$ of sequences indexed by $\Z$ defines a Calkin space $E(\cM,\tau)$ of $\tau$-measurable operators affiliated with a semifinite von Neumann algebra $\cM$ equipped with a faithful normal semifinite trace $\tau$. Furthermore, we show that shift-invariant functionals on $E(\Z)$ generate symmetric functionals on $E(\cM,\tau)$. In the special case, when the algebra $\cM$ is atomless or atomic with atoms of equal trace, the converse also holds and we have a bijective correspondence between all shift-monotone spaces $E(\Z)$ and Calkin spaces $E(\cM,\tau)$ as well as a bijective correspondence between shift-invariant functionals on $E(\Z)$ and symmetric functionals on $E(\cM,\tau)$. The bijective correspondence $E(\Z)\leftrightarrows E(\cM,\tau)$ extends to a correspondence between complete symmetrically $\Delta$-normed spaces $E(\cM,\tau)$ and complete $\Delta$-normed shift-monotone spaces $E(\Z)$. 
\end{abstract}

\maketitle

\section{Introduction}

Let $H$ be a complex separable Hilbert space and let $\cM$ be an atomless (or atomic)  semifinite von Neumann algebra acting on $H$ equipped with a faithful normal semifinite trace $\tau$. Let $S(\cM,\tau)$ be the $*$-algebra of all $\tau$ measurable operators affiliated with $\cM$ (see Section \ref{prel} for the precise definition). In the special case when $\cM=B(H)$ is the type $I$ von Neumann algebra of all bounded operators on $H$ with $\tau$ given by the classical trace $\tr$, the algebra $S(\cM,\tau)$ coincides with the algebra $B(H)$ itself.

The classical Calkin correspondence states that there is a bijective correspondence $E(0,\infty)\leftrightarrows E(\cM,\tau)$ between commutative Calkin function spaces $E(0,\infty)$ on $(0,\infty)$ and noncommutative Calkin spaces $E(\cM,\tau)$ of $\tau$-measurable operators (see Definition \ref{Calkin_dfn}). Denoting by $\mu(A)$ the generalised singular value function of a $\tau$-measurable operator $A$, this correspondence is given by 
\begin{align*}
	E(\cM,\tau)&:=\{X\in \smt: \mu(X)\in E(0,\infty)\},\\
	E(0,\infty)&:=\{f\in S(0,\infty): \mu(f)=\mu(X) \text{ for some } X\in E(\cM,\tau)\},
\end{align*}
where $S(0,\infty)$ is the commutative analogue of $S(\cM,\tau)$ (see \eqref{S}).
In the special case, when $\cM=B(H)$ this correspondence was introduced in \cite{Calkin}.

For the case when $\cM=B(H)$, A. Pietsch suggested in \cite{P_trI} a new approach to the construction of (two-sided) ideals of compact operators and traces on these ideals. Namely, based on his ideas going back to 80th \cite{P90}, Pietsch introduced so-called shift-monotone sequence spaces $E(\Z_+)$, defined corresponding two-sided ideals by setting 
$$E(H):=\{X\in B(H): \{\mu(2^n, X)\}_{n\geq 0}\in E(\Z_+)\},$$
and showed that this correspondence between shift-monotone spaces and two-sided ideals of compact operators is bijective. 
A. Pietsch in \cite{P_trI} also showed that the correspondence $E(H)\leftrightarrows E(\Z_+)$ extends to a correspondence of quasi-norms on $E(H)$ and $E(\Z_+)$. It was also shown in \cite{LPSZ} that this correspondence preserves completeness. 

In the present paper we establish semifinite version of Pietsch correspondence for two-sided ideals and traces on them. 
In particular, we answer affirmatively  A. Pietsch's question from the conference `Singular traces and their application' (Luminy, 2012).
We show in Theorem \ref{thm_corr_nc} that there is a bijective correspondence between Calkin spaces $E(\cM,\tau)$ and shift-monotone spaces $E(\Z)$ of sequences indexed by $\Z$. 
We also prove semifinite version of the  correspondence of $\Delta$-norms on the associated spaces $E(\cM,\tau)\leftrightarrows E(\Z)$ (see Theorem \ref{thm_complete}) which preserves completeness too.

The primary aim of the present paper is the construction of traces and symmetric functionals on Calkin spaces. This story was instigated by J. Dixmier in 1966, who constructed the first example of singular (that is, non-normal) traces (that is, unitarily invariant linear functionals) on Lorentz spaces in $B(H)$ \cite{D}. We refer to the book \cite{LSZ} for detailed information. This traces (termed Dixmier traces) have found applications in the A.~Connes Noncommutative Geometry \cite{C_book}. They play the role of the integral in the noncommutative context and are used to recover the dimension, measure and other geometric characteristics of manifolds (see recent survey \cite{LSZ_survey}).

Calkin correspondence has been successfully used to construct symmetric functionals on $E(\cM,\tau)$. Recall that a linear functional $\phi$ on $E(\cM,\tau)$ is said to be symmetric if $\phi(X)=\phi(Y)$ for all $0\leq X,Y\in E(\cM,\tau)$ with $\mu(X)=\mu(Y)$. If $E(0,\infty)\leftrightarrows E(\cM,\tau)$ are associated spaces, then (see e.g. \cite[Theorem 4.4.1]{LSZ}) any symmetric functional $\lambda$ on $E(0,\infty)$ gives rise to a symmetric functional on $E(\cM,\tau)$ by the formula
\begin{equation}\label{intro_trace_construction}
	\phi(X)=\lambda(\mu(X)), \quad 0\leq X\in E(\cM,\tau).
\end{equation}
However, the proof that the functional $\phi$ is linear is highly nontrivial and relies on recently introduced notion of uniform submajorisation \cite{KS08b}. In the particular case, when $\cM=B(H)$, this construction gives all traces  on a two-sided ideal $E(H)$, since for this algebra any trace is necessarily symmetric functional. In general, this is not the case (see Remark \ref{rem_from_trace_not_shift} below).

A significant advantage of Pietsch construction is that any trace on $E(H)$ can be easily constructed from a shift-monotone functional on $E(\Z_+)$.
Namely, \cite{P_trIII} A. Pietsch showed that an operator $X\in B(H)$ belongs to an ideal $E(H)$ if and only if there exists so-called $E(\Z_+)$-dyadic representation $X=\sum_{k=0}^\infty X_k$, with $X_k$  being finite-rank operators with $\rank(X_k)\leq 2^k$ and
$\Big\{\Big\|X-\sum_{k=0}^nX_k\Big\|_\infty\Big\}_{n\in\Z_+}\in E(\Z)$. 
Then for a $\frac12 S_+$-invariant functional $\theta$ be on $E(\Z_+)$, where $S_+$ is the right-shift operator, setting
$$\phi(X)=\theta\Big\{\frac1{2^k}\tr(X_k)\Big\}_{k\in\Z_+}, \quad X\in E(H),$$
where $X=\sum_{k=0}^\infty X_k$ is an $E(\Z_+)$-dyadic representation of $X$, one can define a bijective correspondence between all $\frac12 S_+$-invariant functionals on $E(\Z_+)$ and all trace $\phi$ on $E(H)$. The proof of this result is elegant and relies only on classical tools of operator theory. 

A slight modification of A. Pietsch approach allows to correspond traces on the weak-trace class ideal in $B(\mathcal H)$ to shift invariant functionals $l_\infty$~\cite{SSUZ}. This approach enabled the solution of several open problems in the theory of singular traces.

In this paper we introduce a similar notion of $E(\Z)$-dyadic representation $X=\sum_{k\in\Z} X_k$ of an operator $X\in E(\cM,\tau)$ and prove in Corollary \ref{cor_one-to-one} that the formula 
$$\phi(X)=\theta\{\frac1{2^k} \tau(X_k)\}_{k\in\Z},$$
defines a bijective correspondence between all symmetric functionals $\phi$ on $E(\cM,\tau)$ and all $\frac 12 S_+$-invariant functionals $\theta$ on $E(\Z)$ (see Corollary \ref{cor_one-to-one} below).

In the case of an arbitrary von Neumann algebra there are two types of symmetric functionals: symmetric functionals supported at infinity and symmetric functionals supported at zero (see Definition \ref{SF}). The first examples of (Dixmier) symmetric functionals supported at zero were given in \cite{GI1995} (see also \cite{DPSS}). Whereas symmetric functionals supported at infinity are related to (local) dimension in the noncommutative formulation, the ones supported at zero are related to (asymptotic) dimension and Novikov-Shubin numbers \cite{GI2000}.

The advantage of our approach (compared to that of~\cite{DPSS} and \cite{GI1995}) is that we are able to construct all symmetric functionals by a single formula. Specialising this formula we obtain symmetric functionals supported at infinity and those supported at zero.

The structure of the paper is the following. In Section \ref{prel} we gather all necessary preliminaries on  the theory of noncommutative integration. In Section \ref{sec_Pietsch_corr} we introduce the notion of shift-monotone sequence spaces on $\Z$ and prove that there is a bijective correspondence between these spaces and Calkin spaces on a semifinite atomless (or atomic) von Neumann algebras. In Section \ref{sec_Pie_corr_traces} we prove bijective correspondence between all symmetric functionals on $E(\cM,\tau)$ and all $\frac12 S_+$-invariant functionals on $E(\Z)$ for associated spaces $E(\cM,\tau)\leftrightarrows E(\Z)$. In Section \ref{sec_norms} we extend the correspondence $E(\cM,\tau)\leftrightarrows E(\Z)$ to a correspondence of $\Delta$-norm of respective spaces and show that this correspondence preserves completeness.

\section{Preliminaries}\label{prel}

In this section, we recall main notions of the theory of noncommutative integration.

In what follows,  $H$ is a complex separable Hilbert space and $B(H)$ is the
$*$-algebra of all bounded linear operators on $H$, equipped with the uniform norm $\|\cdot\|_\infty$, and
$\mathbf{1}$ is the identity operator on $H$. Let $\mathcal{M}$ be
a semifinite von Neumann algebra on $H$, equipped with a faithful normal semifinite trace $\tau$. In the case when $\tau$ is  a finite trace, we assume that it is a state, that is $\tau(\mathbf1)=1$. 

For details on von Neumann algebra
theory, the reader is referred to e.g. \cite{Dix, KR1, KR2}
or \cite{Ta1}. General facts concerning measurable operators may
be found in \cite{Ne, Se} (see also \cite[Chapter
IX]{Ta2} and the forthcoming book \cite{DPS}). For the convenience of the reader, some of the basic definitions are recalled here.

\subsection{The algebra of $\tau$-measurable operators}

A linear operator $X:\mathfrak{D}\left( X\right) \rightarrow H $,
where the domain $\mathfrak{D}\left( X\right) $ of $X$ is a linear
subspace of $H$, is said to be {\it affiliated} with $\mathcal{M}$
if $YX\subseteq XY$ for all $Y\in \mathcal{M}^{\prime }$, where $\mathcal{M}^{\prime }$ is the commutant of $\mathcal{M}$. A linear
operator $X:\mathfrak{D}\left( X\right) \rightarrow H $ is termed
{\it measurable} with respect to $\mathcal{M}$ if $X$ is closed,
densely defined, affiliated with $\mathcal{M}$ and there exists a
sequence $\left\{ P_n\right\}_{n=1}^{\infty}$ in the lattice $P\left(\mathcal{M}\right)$  of all
projections of $\mathcal{M}$, such
that $P_n\uparrow \mathbf{1}$, $P_n(H)\subseteq\mathfrak{D}\left(X\right) $
and $\mathbf{1}-P_n$ is a finite projection (with respect to $\mathcal{M}$)
for all $n$. It should be noted that the condition $P_{n}\left(
H\right) \subseteq \mathfrak{D}\left( X\right) $ implies that
$XP_{n}\in \mathcal{M}$. The collection of all measurable
operators with respect to $\mathcal{M}$ is denoted by $S\left(
\mathcal{M} \right).$ It is a unital $\ast $-algebra
with respect to strong sums and products (denoted simply by $X+Y$ and $XY$ for all $X,Y\in S\left( \mathcal{M%
}\right) $).

Let $\mathcal{M}$ be a
semifinite von Neumann algebra equipped with a faithful normal
semifinite trace $\tau$.
An operator $X\in S\left( \mathcal{M}\right) $ is called $\tau-$measurable if there exists a sequence
$\left\{P_n\right\}_{n=1}^{\infty}$ in $P\left(\mathcal{M}\right)$ such that
$P_n\uparrow \mathbf{1},$ $P_n\left(H\right)\subseteq \mathfrak{D}\left(X\right)$ and
$\tau(\mathbf{1}-P_n)<\infty $ for all $n.$ The collection of all $\tau $-measurable
operators is a unital $\ast $-subalgebra of $S\left(
\mathcal{M}\right) $ and is denoted by $S\left( \mathcal{M}, \tau\right)
$. It is well known that an operator $X$ belongs to $S\left(
\mathcal{M}, \tau\right) $ if and only if $X\in S(\mathcal{M})$
and there exists $\lambda>0$ such that $\tau(E^{|X|}(\lambda,
\infty))<\infty$, where $E^{Y}(a,\infty), a\in\R,$ denotes the spectral projection of a self-adjoint operator $Y\in S(\cM,\tau)$. Alternatively, an operator $X$
affiliated with $\mathcal{M}$ is  $\tau-$measurable (see
\cite{FK}) if and only if
$$\tau\left(E^{|X|}\bigl(n,\infty\bigr)\right)\to 0,\quad n\to\infty.$$

Now, let us recall the definition of the measure topology on $S(\cM,\tau)$. For $0<\varepsilon ,\delta \in \mathbb{R}$, the set
$V\left( \varepsilon ,\delta \right)$ is defined as follows:
\begin{equation*}
	V\left( \varepsilon ,\delta \right) =\left\{ X\in S\left(\cM, \tau \right)
	:\exists \;P\in P\left( \mathcal{M}\right) \text{ s. t. }\left\Vert
	XP\right\Vert _{\infty}\leq \varepsilon ,\tau \left( P^{\bot
	}\right) \leq \delta \right\} \text{.}
\end{equation*}%
The sets $\left\{ V\left( \varepsilon ,\delta \right)
:\varepsilon ,\delta >0\right\} $ form neighbourhood base at zero for the
measure topology $t_\tau$ on $S(\cM,\tau)$. It is known, that $(S(\cM,\tau),t_\tau)$ is a complete Hausdorff topological $*$-algebra.

For any closed and densely defined linear operator $X$
the \emph{null projection} $n(X)=n(|X|)$ is the projection onto its kernel $\mbox{Ker} (X)$, the \emph{range projection } $r(X)$ is the projection onto the closure of its range $\mbox{Ran}(X)$ and the \emph{support projection} $s(X)$ of $X$ is defined by $s(X) ={\bf{1}} - n(X)$. Note that the support projections is the smallest projection in $\cM$, such that $Xs(X)=X.$  Note that for $X,Y\in \smt$, we have that $s(X+Y)\leq s(X)\vee s(Y)$, where $p\vee q$ denotes the supremum of two projections $p,q\in P(\cM)$.
The two-sided ideal $\cF(\cM,\tau)$ in $\cM$ consisting of all elements of $\tau$-finite rank is defined by setting
$$\cF(\cM,\tau) =\{X \in \cM: \tau(r(X)) <\infty \} =\{X \in \cM: \tau(s(X)) <\infty \}. $$

\subsection{Singular value function and its properties}
From now on we assume that $\mathcal M$ is a semifinite von Neumann  algebra $\mathcal M$ equipped
with a faithful normal semi-finite trace $\tau$.

The generalized singular value function $\mu(X):t\rightarrow \mu(t,X)$, $t>0$,  of an operator $X\in \smt$
is defined by setting
\begin{equation}\label{def_mu}
	\mu(t,X)=\inf\{\|XP\|:\ P=P^*\in\mathcal{M}\mbox{ is a projection,}\ \tau(\mathbf{1}-P)\leq t\}.
\end{equation}

There exists an equivalent definition which involves the
distribution function of the operator $X$. For every self-adjoint
operator $X\in S(\mathcal{M},\tau),$ setting
$$d_X(t)=\tau(E^{X}(t,\infty)),\quad t>0,$$
we have (see e.g. \cite{FK} and \cite{LSZ})
$$\mu(t,X)=\inf\{s\geq0:\ d_{|X|}(s)\leq t\}.$$
For $X,Y\in S(\cM,\tau)$ we have
\begin{equation}\label{sum_rear}
	\mu(2t,X+Y)\leq \mu(t,X)+\mu(t,Y), \quad t>0.
\end{equation}

In the special case, when $\mathcal{M}=L_\infty(0,\infty)$ is the von Neumann algebra of all
Lebesgue measurable essentially bounded functions on $(0,\infty)$ acting via multiplication on the Hilbert space
$\mathcal{H}=L_2(0,\infty)$, with the trace given by integration
with respect to Lebesgue measure $m$, the algebra $S(\cM,\tau)$ can be identified with the algebra
\begin{equation}\label{S}
	S(0,\infty)=\{f \text{ is measurable:} \ \exists A\in\Sigma, m((0,\infty)\setminus A)<\infty, f\chi_A\in L_\infty(0,\infty)\},
\end{equation}
where $\chi_{A}$ denotes the characteristic function of a set $A\subset(0,\infty)$.
In this case, the measure topology on $S(\cM,\tau)$ corresponds to the usual topology of convergence in measure on $S(0,\infty)$. The singular value function $\mu(f)$ defined above is precisely the
decreasing rearrangement $f^*$ of the function $f\in S(0,\infty)$ given by
$$f^*(t)=\inf\{s\geq0:\ m(\{|x|\geq s\})\leq t\}.$$

If $\mathcal{M}=B(H)$ and $\tau$ is the
standard trace ${\rm Tr}$, then it is not difficult to see that
$S(\mathcal{M})=S(\mathcal{M},\tau)=\mathcal{M}$ and the measure topology on $S(\cM)=\cM$ coincides with the operator norm topology on $\cM$. In addition,
for $X\in S(\mathcal{M},\tau)$ we have
$$\mu(t,X)=\mu(n,X),\quad t\in[n,n+1),\quad  n\geq0.$$
The sequence $\{\mu(n,X)\}_{n\geq0}$ is just the sequence of singular values of the operator $X.$

The
following definition, originally introduced in \cite{KS08b}, plays a
major role in our constructions.

\begin{defn}\label{uniform majorization def} Let $\mathcal{M}$ be a semifinite von Neumann algebra and let
	$X,Y\in S(\mathcal{M},\tau)$. We say that that $Y$ is uniformly majorized by $X$ (written $Y\lhd X$)
	if there exists $\lambda\in\mathbb{N}$ such that
	\begin{equation}\label{uniform majorization func}
		\int_{\lambda a}^b\mu(s,Y)ds\leq\int_a^b\mu(s,X)ds,\quad 0\leq\lambda a\leq b.
	\end{equation}
\end{defn}
We also use notation $X\lhd f$ if $\mu(X)\lhd \mu(f)$, $X\in S(\cM,\tau), \, f\in S(0,\infty)$.

For $s>0$ we define the dilation operator $\sigma_s$ on $S(0,\infty)$ by setting
$$(\sigma_s(f))(t)=f\bigl(\frac{t}{s}\bigr),\quad t>0.$$

\begin{lem}\cite[Lemma 3.5.4]{LSZ}\label{un_maj} Let 
	$X_k\in S(\mathcal{M},\tau)$, $k\in\mathbb{N}$ and assume that the series $\sum_{k=1}^{\infty}X_k$ converges in the measure topology on $S(\mathcal{M},\tau)$ and the series $\sum_{k=1}^{\infty}\sigma_{2^k}\mu(X_k)$ converges in measure in
	$S(0,\infty)$. Then, with $\lambda=2$ in the definition of uniform majorisation, we have
	$$\sum_{k=1}^{\infty}X_k\lhd2\sum_{k=1}^{\infty}\sigma_{2^k}\mu(X_k).$$
\end{lem}

\subsection{Calkin spaces and symmetric functionals}

Following \cite[Section 2.4]{LSZ} we introduce the notion of Calkin spaces.

\begin{defn}\label{Calkin_dfn}
	A linear subspace $E(\cM,\tau)$ of $\smt$ is called a \emph{Calkin space} if $B\in E(\cM,\tau)$, $A\in\smt$ and $\mu(A)\leq \mu(B)$ implies that $A\in E(\cM,\tau)$. In the special case when $\cM=L_\infty(0,\infty)$ or $\cM=L_\infty(0,1)$ (respectively, $\cM=\ell_\infty(\N)$) we use the term Calkin function (respectively, sequence) space instead  denote $E(\cM,\tau)$ by $E(0,\infty)$ (respectively, $E(0,1),$ or $E(\N)$). 
\end{defn}

Note that any Calkin space $E(\cM,\tau)$ is an $\cM$-bimodule, that is, if $X\in E(\cM,\tau)$, $A,B\in \cM$, then $AXB\in E(\cM,\tau)$. 

As an example of a Calkin space one can take the space $S_0(\cM,\tau)$ of $\tau$-compact operators, defined as 
$$S_0(\cM,\tau) = \{X\in S(\cM,\tau) :  \ \mu(\infty, X)=\lim_{t\to\infty} \mu(t,X)=0\}.$$
Another example of Calkin spaces are the noncommutative $\cL_p$-spaces, defined as 
$$\cL_p(\cM,\tau):=\{X\in \smt: \mu(X)\in L_p(0,\infty)\},\quad 0<p<\infty,$$
where $L_p(0,\infty)$ are the classical Lebesgue $L_p$-spaces.

Next we introduce the main objects of interest in the present paper, traces and symmetric functionals on Calkin spaces. 

\begin{defn}Let $E(\cM,\tau)$ be a Calkin space. A linear functional $\phi$ on $E(\cM,\tau)$ is said to be
	\begin{enumerate}
		\item  a \emph{trace}, if $\phi(UXU^*)=\phi(X)$ for any $X\in E(\cM,\tau)$ and any unitary $U\in \cM$. 
		\item a \emph{symmetric functional}, if $\phi(X)=\phi(Y)$ for any $0\leq X,Y\in E(\cM,\tau)$ with $\mu(X)=\mu(Y).$
	\end{enumerate}
\end{defn}

Is is clear that any symmetric functional on a Calkin space is a trace. However, in general, there are traces, which are not symmetric functionals (see e.g. Remark \ref{rem_from_trace_not_shift}). In the special case, when $\cM=B(H)$ any trace is necessarily a symmetric functional (see e.g. \cite[Lemma 4.5]{GI1995} and  \cite[Lemma 2.7.4]{LSZ}).


\section{Pietsch correspondence for Calkin spaces}\label{sec_Pietsch_corr}

In this section, following \cite{P_trI}, we introduce shift-monotone spaces of sequences, indexed by $\Z$ and show that there is a bijective correspondence between all Calkin spaces and all shift-monotone spaces. 

\subsection{Shift-monotone sequence spaces} 
Let $\eli$ be the algebra of all bounded sequences indexed by $\Z$ and let $S(\Z)$ be the algebra of all two-sided sequences bounded at $+\infty$, that is, all sequences $\{x_n\}_{n\in\Z}$ such that we have $\sup_{k\geq n}|x_k|<\infty$ for all $n\in\Z$. Similarly, we can define the spaces $S(\Z_-)$ and $S(\Z_+)$. We have $S(\Z_+)=\ell_\infty(\Z_+).$

Following the definition of A.Pietsch \cite{P_trI}, for $x\in S(\Z)$ we define the \emph{ordering number} $o_n(x)$, $n\in\Z$, by setting
\begin{equation}
	o_n(x)=\sup_{k\geq n}|x_k|,\quad n\in\Z,
\end{equation}
and use the notation $o(x)=\{o_n(x)\}_{n\in\Z}$.
If $x\in S(\Z)$ is a positive decreasing sequence, then $o(x)=x$.

On the space $S(\Z)$ we define the (right) shift operator $S_+:S(\Z)\to S(\Z)$, by setting
$$(S_+x)_n=x_{n-1},\quad x=\{x_n\}_{n\in\Z}\in S(\Z).$$
Similarly for the spaces $S(\Z_-)$ and $S(\Z_+)$ we define the corresponding right shift operators as follows 
$$S_+\{x_n\}_{n\in\Z_+}=\{0, x_1,x_2,\dots\}, \quad S_+\{x_n\}_{n\in\Z_-}=\{\dots, x_{n-1}, \dots, x_{-3}, x_{-2}\}.$$

We want to associate with a Calkin space $E(\cM,\tau)$ on a von Neumann algebra a linear subspace in $S(\Z)$. To this end, we introduce the notion of shift-monotone spaces on $\Z$.

\begin{defn}\label{def_si}
	A linear subspace $E(\Z)\subset S(\Z)$ is called a \emph{shift-monotone space} if
	\begin{enumerate}
		\item $x\in S(Z), y\in E(\Z)$ and $o(x)\leq o(y)$ implies that $x\in E(\Z).$ 
		\item $S_+x\in E(\Z)$ for any $x\in E(\Z)$. 
	\end{enumerate}
	Similarly, one can define shift-monotone spaces on $\Z_-=\{n\in\Z, z\leq 0\}$ and on $\Z_+=\{n\in\Z:n\geq 0\}$ (in the latter case, one recovers the original definition due to A.Pietsch \cite{P_trI}).
\end{defn}

We note that any shift-monotone space $E(\Z)$ is an $\ell_\infty(\Z)$-bimodule with respect to the pointwise operations in $S(\Z)$, since for any $a, b\in\ell_\infty(\Z), x\in E(\Z)$ we have
$$o_n(axb)=\sup_{k\geq n}|a_nx_nb_n|\leq \|a\|_\infty\|b\|_\infty o_n(x),\quad n\in\Z.$$

In particular, any shift-monotone space $E(\Z)$ is a solid subspace of $S(\Z)$, that is if $x\in E(\Z)$, $y\in S(\Z)$ and $|y|\leq |x|$, then $y\in E(\Z)$.

\subsection{Pietsch correspondence for commutative case}

In order to introduce the Pietsch correspondence between Calkin function spaces in $S(0,\infty)$
and shift-monotone spaces in $S(\Z)$
we define the Pietsch operator  $D:S(\Z)\to S(0,\infty)$
by setting
\begin{equation}\label{def_D}
	Dx=\sum_{n\in\Z} x_n\chi_{[2^n,2^{n+1})},\quad x\in S(\Z).
\end{equation}
The operator $D$ is well-defined since any $x\in S(\Z)$ is bounded at $+\infty$ and so  $Dx\cdot \chi_{[1,\infty)}$ is bounded.
It is clear that $D$ is a linear and positive from $S(\Z)$ into $S(0,\infty)$.

Recall (see e.g. \cite{Kalton}), that the decreasing rearrangement $f^*$ of a function $f\in S(0,\infty)$ may be also defined by
\begin{equation}\label{f^*_equiv_def}
	f^*(t)=\inf_{m(B)< t}\sup_{s\notin B}|f(s)|.
\end{equation}

The following lemma is analogous to \cite[Lemma 4.4]{P_trI}.
\begin{lem}\label{order_vs_rear}
	For any $x\in S(\Z)$
	we have that
	$$o_{n}(x)=(Dx)^*(2^n),\quad n\in\Z.$$
	In particular,
	$$(Dx)^*\leq D(o(x))\leq \sigma_2 (Dx)^*,\quad x\in S(\Z).$$
\end{lem}
\begin{proof}
	Let $n\in \Z$ and let $B$ be a measurable subset of $(0,\infty)$ with $m(B)< 2^n$.  We have
	$$\sup_{s\notin B}|Dx|(s)=\sup_{s\notin B}\sum_{k\in\Z}|x_k|\chi_{[2^k,2^{k+1})}(s)\geq \sup_{k\geq n}|x_k|=o_{n}(x).$$
	Therefore, using \eqref{f^*_equiv_def}, we infer
	$$(Dx)^*(2^n)=\inf_{m(B)< 2^n}\sup_{s\notin B}|Dx|(s)\geq o_{n}(x).$$
	
	On the other hand,
	\begin{align*}
		(Dx)^*(2^n)&=\inf_{m(B)< 2^n}\sup_{s\notin B}|Dx|(s)\leq \sup_{s\notin (0,2^n)}\sum_{k\in\Z}|x_k|\chi_{[2^k,2^{k+1})}(s)\\
		&=\sup_{k\geq n}|x_k|=o_n(x),
	\end{align*}
	which proves that
	$$(Dx)^*(2^n)= o_{n}(x).$$
	
	Further, since
	\begin{equation*}\label{f^*_sigmaf^*}
		f^*(t)\leq f^*(2^n)\leq \sigma_2(f^*(t)), \quad f\in S(0,\infty)
	\end{equation*}
	for all $t\in [2^n, 2^{n+1})$ and $D(o(x))(2^n)=o_n(x)$ we have that
	$$(Dx)^*(t)\leq (Dx)^*(2^n)= o_n(x)=D(o(x))(t), \ t\in [2^n, 2^{n+1})$$
	and
	$$D(o(x))(t)=o_n(x)=(Dx)^*(2^n)\leq (Dx)^*(t/2)=\sigma_2(Dx)^*(t), \ t\in [2^n, 2^{n+1}).$$
	Since the estimates above hold for an arbitrary $n\in \Z$, we infer that
	$$(Dx)^*\leq D(o(x))\leq \sigma_2 (Dx)^*.$$
	
\end{proof}

So far we introduced the operator $D$ which maps sequences into functions. We introduce now `almost' inverse $\Phi:S(0,\infty)\to S(\Z)$ of this operator by setting
$$\Phi f:=\{f^*(2^n)\}_{n\in\Z}, \quad f\in S(0,\infty).$$
Note that inequality \eqref{sum_rear} implies that for any $f,g\in S(0,\infty)$ we have 
\begin{align}\label{eq_Phi_shift}
	\Phi(f+g)=\{(f+g)^*(2^n)\}\leq \{f^*(2^{n-1})\}+\{g^*(2^{n-1})\}=S_+(\Phi f +\Phi g).
\end{align}

The following lemma shows that $\Phi$ is indeed `almost'  inverse of $D$.

\begin{lem}\label{exp_f_f}
	\begin{enumerate}
		\item Let $E(0,\infty)$ be a Calkin space. 
		A  function $f\in S(0,\infty)$ belongs to $E(0,\infty)$ if and only if $D\Phi f\in E(0,\infty)$. 
		
		\item 
		For every $x\in S(\Z)$ we have
		$\Phi Dx=o(x).$ In particular, if $E(\Z)$ is a shift-monotone space, then $x\in S(\Z)$ belongs to $E(\Z)$ if and only if $\Phi Dx\in E(\Z).$
	\end{enumerate}
\end{lem}
\begin{proof}

	(i) Without loss of generality $f=f^*\in S(0,\infty)$. 
	
	Since $f$ is decreasing, it follows that 
	\begin{equation}\label{f_via_expec}
		f\leq \sum_{n\in \Z}f(2^n)\chi_{[2^n,2^{n+1})}=D\Phi f.
	\end{equation}
	Hence, if $D\Phi f\in E(0,\infty)$, then $f\in E(0,\infty)$.
	
	Conversely, assume $f\in E(0,\infty)$. We have 
	\begin{equation}\label{expec_via_f}
		D\Phi f=\sum_{n\in \Z}f(2^n)\chi_{[2^n,2^{n+1})}=\sum_{n\in \Z}\sigma_2 f(2^{n+1})\chi_{[2^n,2^{n+1})}\leq \sigma_2 f,
	\end{equation}
	and so $D\Phi f\in E(0,\infty)$. 
	
	(ii) For every $x\in S(\Z)$ we have
	$$(\Phi Dx)_n=(Dx)^*(2^n)=o_n(x),$$
	by Lemma \ref{order_vs_rear}. The assertion follows from the fact that $E(\Z)$ is shift-monotone space.

\end{proof}

\begin{thm}[Pietsch correspondence in commutative case]\label{thm_corr_commutative}
	The rule 
	\begin{align*}
		\pi:E(0,\infty)\to E(\Z):=\{x\in S(\Z): Dx\in E(0,\infty)\},
	\end{align*}
	defines a one-to-one correspondence between the Calkin function spaces on $(0,\infty)$ and shift-monotone sequence spaces in $S(\Z)$. The inverse is given by 
	\begin{align*}
		\pi^{-1}:E(\Z)\to E(0,\infty):=\{f\in S(0,\infty): \Phi f\in E(\Z)\},
	\end{align*}
\end{thm}
\begin{proof}
	We firstly show that $\pi$ is well-defined. Let $E(0,\infty)$ be a Calkin function space.  The fact that $E(\Z)$ is a linear subspace of $S(\Z)$ follows from the linearity of the operator $D$. 
	
	Let $x\in E(\Z)$ and $y\in S(\Z)$ be such that $o(y)\leq o(x)$ for all $n\in \Z$.
	By Lemma~\ref{order_vs_rear} we have
	\begin{equation}\label{Dy_Dox}
		(Dy)^*\leq D(o(y))\leq D(o(x))\leq \sigma_2(Dx)^*.
	\end{equation}
	Since 
	$\sigma_2(Dx)^*\in E(0,\infty)$, it follows that $Dy\in E(0,\infty)$, or equivalently $y\in  E(\Z)$.
	
	Further, for $x\in  E(\Z)$, we have 
	$$D(S_+x)=\sum_{n\in\Z}x_{n-1}\chi_{[2^n,2^{n+1})}=\sigma_2 \Big(\sum_{n\in\Z}x_{n-1}\chi_{[2^{n-1},2^{n})}\Big)=\sigma_2(Dx).$$
	Since $\sigma_2(Dx)\in E(0,\infty)$, it follows that  $D(S_+x)\in E(0,\infty)$, and so $S_+x\in  E(\Z)$. 
	Thus, $E(\Z)$ is a shift-monotone sequence space.

	Next, we show that the rule $\pi^{-1}$ is also well-defined. Let $E(\Z)$ be a shift-monotone sequence space. We claim that the set $E^*(0,\infty)=\{f\in S(0,\infty): \Phi f\in E(\Z)\}$ is a Calkin space.  Using the inequality \eqref{eq_Phi_shift} for  $f,g\in S(0,\infty)$, we have $\Phi(f+g)\leq S_+(\Phi f+\Phi g).$ Since $E(\Z)$ is shift-monotone, it follows that $\Phi(f+g)\in E(\Z)$, that is $f+g\in E^*(0,\infty)$. Thus, $E^*(0,\infty)$ is linear subspace of $\smt$. 
	Now, let $f\in S(0,\infty)$, $g\in E^*(0,\infty)$ and $f^*\leq g^*$. Definition of $\Phi$ immediately implies that $\Phi f \leq \Phi g$, and so $\Phi f\in E(\Z)$, or, equivalently, $f\in E^*(0,\infty)$.  Thus,  $E^*(0,\infty)$ is a Calkin function space. 
	
	Finally, we show that $\pi^{-1}$ is indeed the inverse of $\pi$. If $E(0,\infty)\xrightarrow{\pi} E(\Z)\xrightarrow{\pi^{-1}} E^*(0,\infty)$, then by Lemma \ref{exp_f_f} we have 
	\begin{align*}
		E^*(0,\infty)&=\{f\in S(0,\infty): \Phi f\in  E(\Z)\}\\
		&=\{f\in S(0,\infty): D\Phi f\in E(0,\infty)\}=E(0,\infty).
	\end{align*}
	Thus, $\pi^{-1}$ is the left inverse of $\pi$ on the collection of all Calkin  function spaces. 
	
	On the other hand, if  $E(Z)\xrightarrow{\pi^{-1}}E^*(0,\infty)\xrightarrow{\pi} E^*(\Z),$ 
	then by Lemma \ref{exp_f_f} (ii) we have
	\begin{align*}
		E^*(\Z)&=\{x\in S(\Z): Dx\in E^*(0,\infty)\}=\{x\in S(\Z): \Phi Dx\in E(\Z)\}\\
		&=\{x\in S(\Z): o(x)\in E(\Z)\}=E(\Z).
	\end{align*}
	Thus, $\pi^{-1}$ is also the right inverse of $\pi$, as required.

\end{proof}

\begin{rem}\label{rem_different_approach}
	Similar assertions hold for Calkin function spaces on $(0,1)$ (respectively,  Calkin sequence spaces on $\N$) and shift-monotone spaces on $\Z_-$ (respectively, on $\Z_+$).

\end{rem}

\subsection{Noncommutative case}

Now, let $\cM$ be an atomless (or atomic) von Neumann algebra equipped with a faithful normal semifinite trace $\tau$. The noncommutative analogue of Theorem \ref{thm_corr_commutative} follows from the classical Calkin correspondence (see e.g. \cite[Chapter 2]{LSZ}).  However, for clarity of the bijective coorespondnce of symmetric functionals on Calkin spaces and $\frac12 S_+$-invariant functionals on the corresponding shift-invariant sequence space, stated in Corollary \ref{cor_one-to-one}  we prefer a self-contained exposition. 

We recall (see \cite{P_trI}) that ideals of compact operators on a Hilbert space $H$, corresponding to a shift-monotone ideal on $\Z_+$, are constructed via diagonal operator with respect to a fixed orthonormal basis in $H$. For atomless von Neumann algebra there is an analogue of diagonal operator given in the following theorem.

\begin{thm}\cite[Theorem 2.3.11]{LSZ}, \cite[Lemma 1.3]{CKS}
	\label{atomless_diagonal} Let $\cM$ be an atomless (or atomic) von Neumann algebra equipped with a faithful normal semifinite trace $\tau$ and let $0\leq A\in S_0(\cM,\tau)$. There exits a $\sigma$-finite commutative
	subalgebra $\cM_0$ of $\cM$ such that the restriction $\tau|_{\cM_0}$ is semifinite. 
	
	Moreover, there exists a trace preserving $*$-isomorphism $\iota$ from $S(0,\tau(\mathbf{1}))$ onto $S(\cM_0,\tau|_{\cM_0})$ (respectively, $\ell_\infty\to \cM_0$), such
	that $\mu(\iota(X))=\mu(X)$ for any $X\in S(0,\tau(\mathbf{1}))$ and $\iota(\mu(A))=A.$
\end{thm}

In the following we choose an operator $0\leq A\in S_0(\cM,\tau)$, and fix the commutative algebra $\cM_0$ and the isomorphism $\iota$.

Now, we shall extend the Pietsch correspondence to general Calkin operator spaces.
Similarly to the case of functions we define the operator $\Phi:S(\cM,\tau)\to S(\Z)$ by setting
$$\Phi X=\{\mu(2^n, X)\}_{n\in\Z}, \quad X\in S(\cM,\tau).$$
The operator $\Phi$ is well-defined since $\mu(2^n,X)\leq \mu(1,X)$ for any $n\in\N$, and so $\Phi X\in S(\Z)$.

The `almost' inverse $\cD:S(\Z)\to S(\cM,\tau)$ of $\Phi$ is defined as follows 
\begin{equation}\label{eq_def_diag}
	\cD x=\iota Dx,\quad x\in S(\Z).
\end{equation}
The mapping $\cD$ is well-defined, since both mappings $D:S(\Z)\to S(0,\tau(\mathbf{1}))$ and $\iota:S (0,\tau(\mathbf{1}))\to S(\cM,\tau)$ are well-defined.

\begin{rem}Note that the choice of an operator $0\leq A\in S_0(\cM,\tau)$, and hence the choice of the `diagonal' algebra $\cM_0$ and the isomorphism $\iota$ in the `diagonal operator' $\cD$ is irrelevant for our construction. 
\end{rem}

\begin{thm}[Pietsch correspondence]
	\label{thm_corr_nc}Let $\cM$ be a semifinite von Neumann algebra equipped with a faithful normal semifinite trace $\tau$. 
	
	\begin{enumerate}
		\item If $\cM$ is atomless  with $\tau(\mathbf{1})=\infty$, then 
		\begin{align*}
			E(\Z)&:=\{x\in S(\Z): \cD x\in E(\cM,\tau)\}\\
			&=\{x\in S(\Z): (Dx)^*=\mu(X) \text{ for some } X\in E(\cM,\tau)\}
		\end{align*}
		and 
		\begin{align*}
			E(\cM,\tau):=\{X\in S(\cM,\tau): \Phi X\in E(\Z)\}
		\end{align*}
		defines a one-to-one correspondence $E(\Z)\leftrightarrows E(\cM,\tau)$ between Calkin  spaces on $\cM$ and shift-monotone sequence spaces in $S(\Z)$. 
		
		\item If $\cM$ is atomless  with $\tau(\mathbf{1})=1$, then 
		\begin{align*}
			E(\Z_-)&:=\{x\in S(\Z_-): \cD x\in E(\cM,\tau)\}\\
			&=\{x\in S(\Z_-): (Dx)^*=\mu(X) \text{ for some } X\in E(\cM,\tau)\}
		\end{align*}
		and 
		\begin{align*}
			E(\cM,\tau):=\{X\in S(\cM,\tau): \Phi X\in E(\Z_-)\}
		\end{align*}
		defines a one-to-one correspondence $E(\Z_-)\leftrightarrows E(\cM,\tau)$ between Calkin operator spaces on $\cM$ and shift-monotone sequence spaces in $S(\Z_-)$. 
		
		\item If $\cM$ is atomic (with all atoms of trace $1$), then 
		\begin{align*}
			E(\Z_+)&:=\{x\in S(\Z_+): \cD x\in E(\cM,\tau)\}\\
			&=\{x\in S(\Z_+): (Dx)^*=\mu(X) \text{ for some } X\in E(\cM,\tau)\}
		\end{align*}
		and 
		\begin{align*}
			E(\cM,\tau):=\{X\in S(\cM,\tau): \Phi X\in E(\Z_+)\}
		\end{align*}
		defines a one-to-one correspondence $E(\Z_+)\leftrightarrows E(\cM,\tau)$ between Calkin operator spaces  on $\cM$ and shift-monotone sequence spaces in $S(\Z_+)$. 
	\end{enumerate}
	
\end{thm}
\begin{proof}
	We give a proof of the case (i) only. Case (ii) can be proved similarly. Case (iii) was proved in \cite[Theorem 4.7]{P_trI}.

	As in the proof of \cite[Theorem 2.5.3]{LSZ} we set $$E(0,\infty):=\iota^{-1}\big(E(\cM,\tau)\cap S(\cM_0,\tau|_{\cM_0})\big).$$ It is clear that $E(0,\infty)$ is a
	Calkin function space.
	In addition, $E(0,\infty)$ can be described as
	$$E(0,\infty)=\{f\in S(0,\infty):f^*=\mu(X) \text{ for some } X\in E(\cM,\tau)\}.$$
	%
	
	Theorem \ref{thm_corr_commutative} yields a bijective correspondence between shift-monotone space $E(\Z)$ and Calkin function space $E(0,\infty)$ with 
	$$E(\Z)=\{x\in S(\Z): Dx\in E(0,\infty)\}.$$
	
	Since 
	\begin{align*}
		&\{x\in S(\Z): \cD x\in E(\cM,\tau)\}\\
		&= \{x\in S(\Z): (Dx)^*=\mu(X) \text{ for some } X\in E(\cM,\tau)\},
	\end{align*}
	it follows that
	\begin{align*}
		E(\Z)&=\{x\in S(\Z): \cD x\in E(\cM,\tau)\}\\
		&= \{x\in S(\Z): (Dx)^*=\mu(X) \text{ for some } X\in E(\cM,\tau)\},
	\end{align*}
	is a shift-monotone sequence space. The converse follows from  Theorem \ref{thm_corr_commutative}.
\end{proof}

\begin{rem}\label{rem_diff_types}
	\begin{enumerate}
		\item 
		In the following we do not distinguish between atomless (both finite and infinite trace) and atomic cases and simply use notations $E(\Z)$ for all three cases in Theorem \ref{thm_corr_nc} making agreement that $S(\Z_-),S(\Z_+)\subset S(\Z)$ using natural identification.

		\item If $\cM$ is an arbitrary (not necessarily atomless or atomic) von Neumann algebra, equipped with a faithful normal semifinite trace $\tau$, one can construct Calkin spaces from shift-monotone spaces in the same way. Namely, if $E(\Z)$ is a shift-monotone space, then 
		$$E(\cM,\tau)=\{X\in \smt:\Phi X\in E(\Z)\},$$
		defines a Calkin space (with the argument verbatim to the proof of Theorem \ref{thm_corr_commutative}). The lack of bijectivity in the Pietsch correspondence in the nontomic, nonatomless case stems from the absence, in general, of a `diagonal' algebra (the algebra $\cM_0$ of Theorem \ref{atomless_diagonal}). 
	\end{enumerate}
\end{rem}

\section{Construction of traces}\label{sec_Pie_corr_traces}

In this section we adapt the method suggested by A. Pietsch \cite{P_trIII} for construction of traces on ideals of compact operator on a Hilbert space to the setting of a semifinite von Neumann algebra. From now on we assume that $E(\Z)\leftrightarrows E(\cM,\tau)$ are associated spaces via Theorem \ref{thm_corr_nc}.

We start with the dyadic representation of an operator (see \cite[Section 4]{P_trIII}). 
\begin{defn}Let $E(\Z)$ be a shift-monotone sequence space and let $X\in \smt$. An infinite series representation 
	$$X=\sum_{k\in\Z} X_k,$$
	where the convergence is understood in the measure topology, 
	is called an $E(\Z)$-dyadic representation of $X$ if $X_k\in \cF(\cM,\tau),$ $\tau(s(X_k))\leq 2^k$ and $$\Big\{\big\| X-\sum_{k=-\infty}^nX_k\big\|\Big\}_{n\in\Z}\in E(\Z).$$
\end{defn}

\begin{lem}Let $X\in \smt$. Then $X\in E(\cM,\tau)$ if and only if there exists an $E(\Z)$-dyadic representation of $X$. 
\end{lem}
\begin{proof}Suppose firstly that $X\in E(\cM,\tau)$.
	We claim that there exists a sequence $\{P_k\}_{k\in\Z}\subset P(\cM)$, such that 
	$$\|X(\textbf{1}-P_k)\|\leq 2\mu(2^k, X),\quad \tau(P_k)\leq 2^k.$$

	Fix $k\in \Z$ and suppose firstly that $\mu(2^k,X)>0$. Assume to the  contrary that for any projection $Q$ with $\tau(Q)\leq 2^k$, we have the following inequality: $\|X(\mathbf{1}-Q)\|> 2\mu(2^k,X)$. Then, by the definition of the singular value function (see \eqref{def_mu}), we have that 
	$$\mu(2^k, X)=\inf \{\|X(\mathbf{1}-Q)\|: Q\in P(\cM), \tau(Q)\leq 2^k\}\geq 2\mu(2^k,X).$$
	Since $\mu(2^k,X)> 0$, we arrive at the contradiction.  On the other hand, if $\mu(2^k,X)=0$, then $\tau(s(X))\leq 2^k$. Therefore, one can take $P_k=s(X)$ since $X(\mathbf{1}-s(X))=0.$ Thus, the existence of the required sequence $\{P_k\}$ is guaranteed.

	We set 
	$$X_k=X(P_{k-1}-P_{k-2}).$$ Since $\tau(P_k)\leq 2^k,$ it follows that $\tau(s(X_k))\leq \tau(P_{k-1})+\tau(P_{k-2})\leq 2^k$, in particular, $X_k\in \cF(\cM,\tau)$. Furthermore, 
	\begin{align*}
		\big\|X-\sum_{k=-\infty}^n X_k\big\|&=\big\|X-\sum_{k=-\infty}^n X(P_{k-1}-P_{k-2})\big\|\\
		&=\big\|X-XP_{n-1}\big\|\leq 2\mu(2^{n-1},X).
	\end{align*}
	
	By assumption $X\in E(\cM,\tau)$, and so, by Theorem \ref{thm_corr_nc} the sequence $\Phi X=\{\mu(2^n, X)\}_{n\in\Z}$ belongs to $E(\Z)$. Since $E(\Z)$ is shift-monotone, it follows that $\{\mu(2^{n-1}, X)\}_{n\in\Z}\in E(\Z)$ and so,
	$$\Big\{\big\| X-\sum_{k=-\infty}^nX_k\big\|\Big\}_{n\in\Z}\in E(\Z),$$
	as required.  
	
	Conversely, let $X=\sum_{k\in\Z} X_k$ be an $E(\Z)$-dyadic representation. 
	By definition of the support projection, we have that $s(\sum_{k=-\infty}^n X_k)\leq \sup_{k\leq n} s(X_k).$ Therefore, the assumption $\tau(s(X_k))\leq 2^k$ implies that 
	$$\tau(s(\sum_{k=-\infty}^n X_k))\leq \tau(\sup_{k\leq n} s(X_k))=\sum_{k\leq n} 2^k\leq 2^{n+1}.$$ 
	Hence, by definition of the singular value function we infer that 
	$\Phi X=\mu(2^n, X)\leq \|X-\sum_{k=-\infty}^n X_k\|,$ and so $\Phi X\in E(\Z)$. By Theorem \ref{thm_corr_nc}, we obtain that $X\in E(\cM,\tau)$. 
	
\end{proof}

The following result is a semifinite analogue of \cite[Lemma2]{P_trIII}.

\begin{lem}\label{lem_sum_dyadic}
	Let $X,Y\in E(\cM,\tau)$  with $E(\Z)$-dyadic representations
	$$X=\sum_{k\in\Z} X_k, \quad Y=\sum_{k\in\Z} Y_k.$$
	Then the operator $X+Y$ has an $E(\Z)$-dyadic representation 
	$$X+Y=\sum_{k\in\Z} X_{k-1}+Y_{k-1}.$$
\end{lem}
\begin{proof}
	The equality $X+Y=\sum_{k\in\Z} X_{k-1}+Y_{k-1}$ follows from the fact that\\ $(\smt, t_\tau)$ is a topological algebra. Since 
	$$\tau(s(X_{k-1}+Y_{k-1}))\leq \tau(s(X_{k-1}) \vee s(Y_{k-1}))\leq \tau(s(X_{k-1}))+\tau(s(Y_{k-1}))\leq 2^k,$$
	and 
	\begin{align*}
		\|X+Y&-\sum_{k=\infty}^n (X_{k-1}+Y_{k-1})\|\}_{n\in\Z}\\
		&\leq \{\|X-\sum_{k=-\infty}^{n-1}X_k\|\}_{n\in\Z}+\{\|Y-\sum_{k=-\infty}^{n-1}Y_k\|\}_{n\in\Z}\\
		&=S_+\{\|X-\sum_{k=-\infty}^{n}X_k\|\}_{n\in\Z}+S_+\{\|Y-\sum_{k=-\infty}^{n}Y_k\|\}_{n\in\Z}\in E(\Z),
	\end{align*}
	it follows that $X+Y=\sum_{k\in\Z} X_{k-1}+Y_{k-1}$ is indeed $E(\Z)$-dyadic representation on $X+Y$. 
\end{proof}

In the following theorem we show that any $\frac12 S_+$-invariant functional on $E(\Z)$ gives rise to a symmetric functional on the corresponding Calkin space $E(\cM,\tau)$. This provides a semifinite version of \cite[Theorem 4]{P_trIII}. 
\begin{thm}\label{thm_Pietsch_trace}Let $E(\Z)\rightleftarrows E(\cM,\tau)$ be associated space such that $E(\cM,\tau)\subset S_0(\cM,\tau)$. Assume that $\theta$ is a $\frac12 S_+$-invariant functional on $E(\Z)$. Define 
	\begin{equation}\label{eq_Pie_trace_from_shift}
		\phi (X)=\theta\Big(\{ \frac1{2^k}\tau(X_k)\}_{k\in \Z}\Big),
	\end{equation}
	where $X=\sum_{k\in\Z} X_k$ is an $E(\Z)$-dyadic representation of $X\in E(\cM,\tau)$. Then $\phi$ is a symmetric functional on the corresponding Calkin space $E(\cM,\tau)$. 
\end{thm}

\begin{rem}The assumption that $E(\cM,\tau)\subset S_0(\cM,\tau)$ is not restrictive. Indeed, if $E(\cM,\tau)$ contains an operator, which is not $\tau$-compact, then since $E(\cM,\tau)$ is a Calkin space, it contains the whole algebra $\cM$ and the trace $\tau$ is necessarily infinite. As shown in \cite[Theorem 5]{SZ_Crelle} there are no symmetric functionals on $\cM$ with infinite trace $\tau$. In particular, it means that there are no symmetric functionals on $E(\cM,\tau)$. 
\end{rem}

\begin{proof}
	We firstly show that $\phi:E(\cM,\tau)\to \C$ is well-defined. Let $X\in E(\cM,\tau)$ be fixed and let  $X=\sum_{k\in\Z} X_k=\sum_{k\in\Z} Y_k,$ be two $E(\Z)$-dyadic representations of $X$. Let $a=\{\frac1{2^k}\tau(X_k-Y_k)\}_{k\in\Z}$. It is sufficient to show that $a\in \mathrm {Range}(1-\frac12 S_+)$.

	We claim firstly that the sequence 
	$$b=\Big\{\frac1{2^k} \tau\Big(\sum_{n=-\infty}^k (X_n-Y_n)\Big)\Big\}_{k\in\Z}$$ 
	belongs to $S(\Z)$. Since $\cL_1(\cM,\tau)$ is an $\cM$-bimodule it is sufficient to show that for every fixed $k\in\Z$, the operator $\sum_{n=-\infty}^k (X_n-Y_n)$ is bounded and has finite support. 
	
	Fix $k\in \Z$. 
	Since $\tau(s(X_k))\leq 2^k$, it follows that
	$$\tau\Big(s\big(\sum_{n=-\infty}^k X_n\big)\Big)\leq \tau\Big(\sup_{n\leq k} s(X_n)\Big)\leq \sup_{n\leq k} 2^n=2^k<\infty,$$

	and similarly, $\tau\Big(s\big(\sum_{n=-\infty}^k Y_n\big)\Big)=2^k<\infty.$ This guarantees that
	\begin{equation}\label{eq_b_fin_supp}
		\tau\left(s\Big(\sum_{n=-\infty}^k (X_n-Y_n)\Big)\right)\leq 2^{k+1}<\infty
	\end{equation}
	Furthermore, since 
	$$\Big\|\sum_{n=-\infty}^k (X_n-Y_n)\Big\|\leq \Big\|X-\sum_{n=-\infty}^k X_n\Big\|+\Big\|X-\sum_{n=-\infty}^k Y_n\Big\|<\infty,$$
	it follows that 
	$\sum_{n=-\infty}^k(X_n-Y_n)\in \cL_1(\cM,\tau)$. Thus, $b$ is well-defined element in $S(\Z)$.

	Moreover, we have estimate
	\begin{align*}
		|b|&= \Big\{\frac1{2^k} \Big|\tau \big(\sum_{n=-\infty}^k (X_n-Y_n)\big)\Big|\Big\}_{k\in\Z}\\
		&\leq \Big\{\frac1{2^k} \Big\|\sum_{n=-\infty}^k (X_n-Y_n)\Big\|\cdot  \tau\Big(s\big(\sum_{n=-\infty}^k X_n-Y_n)\big)\Big)\Big\}_{k\in\Z}\\
		&\stackrel{\eqref{eq_b_fin_supp}}{\leq} 2\Big( \Big\{  \big\|X-\sum_{n=-\infty}^k X_n\big\|\Big\}_{k\in\Z}+\Big\{\big\|X-\sum_{n=-\infty}^k Y_n\big\|\Big\}_{k\in\Z}\Big)\in E(\Z),
	\end{align*}
	that is, $b\in E(\Z)$. Trivial computation shows that $a=b-\frac12 S_+ b$. Hence, 
	$$\theta(a)=\theta(b)-\theta(\frac12 S_+ b)=0,$$
	as required. Thus, $\theta(\{\frac1{2^k}\tau(X_k)\}_{k\in\Z})=\theta(\{\frac1{2^k}\tau(Y_k)\}_{k\in\Z})$, so that $\phi(X)$ is well-defined.

	Next, we show that $\phi$ is indeed a symmetric functional.
	To prove linearity of $\phi$ we note that for any $X,Y\in E(\cM,\tau)$, Lemma \ref{lem_sum_dyadic} implies that $X+Y=\sum_{k\in\Z} (X_{k-1}+Y_{k-1})$ is $E(\Z)$-representation of $X+Y$. Therefore, since $\theta$ is $\frac12 S_+$-invariant, we have  
	\begin{align*}
		\phi(X+Y)&=\theta\Big(\big\{\frac1{2^k}\tau(X_{k-1}+Y_{k-1})\big\}_{k\in\Z}\Big)\\
		&=\theta\Big(\frac 12 S_+\big( \big\{\frac1{2^k}\tau(X_k)\big\}+\big\{\frac1{2^k}\tau(Y_k)\big\}\big)\Big)\\
		&=\theta\Big(\big\{\frac1{2^k}\tau(X_k)\big\}\Big)+\theta\Big(\big\{\frac1{2^k}\tau(Y_k)\big\}\Big)=\phi(X)+\phi(Y).
	\end{align*}
	Thus, $\phi:E(\cM,\tau)\to \C$ is a linear mapping.

	Now, let $0\leq X,Y\in E(\cM,\tau)$ be such that $\mu(X)=\mu(Y)$. Since $E(\cM,\tau)\subset S_0(\cM,\tau)$, it follows that $\mu(\infty, X)=\mu(\infty, Y)=0$. Let $\iota_X$ and $\iota_Y$ be $*$-isomorphisms of Theorem \ref{atomless_diagonal}, such that $\iota_X(\mu(X))=X$ and $\iota_Y(\mu(Y))=Y$. Note that $\mu(X)=\sum_{n\in\Z} \mu(X)\chi_{[2^n, 2^{n+1})}$ is a $E(\Z)$-dyadic decomposition of $\mu(X)\in E(0,\infty)$, since 
	\begin{align*}
		\Big\{\Big\|\mu(X)&-\sum_{n=-\infty}^k \mu(X)\chi_{[2^n, 2^{n+1})}\Big\|\Big\}_{k\in \Z}=\{\|\mu(X)\chi_{[2^{k+1}, \infty)}\|\}_{k\in\Z}\\
		&\leq \{\mu(2^{k+1}, X)\}_{k\in\Z}\in E(\Z).
	\end{align*}
	We set
	$$X_n=\iota_X(\mu(X)\chi_{[2^n,2^{n+1})}), \quad Y_n=\iota_Y (\mu(X)\chi_{[2^n,2^{n+1})}).$$
	Since both $\iota_X$ is trace-preserving, it follows that $$\tau(s(X_n))\leq \tau(s(\iota_X\chi_{[2^n,2^{n+1})}))=2^n.$$ Furthermore, 
	\begin{align*}
		\Big\|X-\sum_{n=-\infty}^k X_n\Big\|&=\mu\Big(0, X-\sum_{n=-\infty}^k X_n\Big)\\
		&=\mu\Big(0, \iota_X\big(\mu(X)-\sum_{n=-\infty}^k \mu(X)\chi_{[2^n, 2^{n+1})}\big)\Big)\\
		&=\mu\Big(0,\mu(X)-\sum_{n=-\infty}^k \mu(X)\chi_{[2^n, 2^{n+1})}\Big)\\
		&=\Big\|\mu(X)-\sum_{n=-\infty}^k \mu(X)\chi_{[2^n, 2^{n+1})}\Big\|,
	\end{align*}
	which implies that $\Big\{\Big\|X-\sum_{n=-\infty}^k X_n\Big\|\Big\}\in E(\Z)$. Thus, $X=\sum_{n\in\Z} X_n$ is a $E(\Z)$-dyadic decomposition of $X$. Similarly, $Y=\sum_{n\in\Z} Y_n$ is  a $E(\Z)$-dyadic decomposition of $Y$.
	
	Using again the fact that both $\iota_X$ and $\iota_Y$ are trace-preserving, it follows that 
	\begin{align*}
		\phi(X)&=\theta\Big(\{\frac1{2^k} \tau(X_n)\}\Big)=\theta\Big(\{\frac1{2^k} \tau(\iota_X(\mu(X)\chi_{[2^n,2^{n+1})})\}\Big)\\
		&=\theta\Big(\{\frac1{2^k} \tau(\iota_Y(\mu(X)\chi_{[2^n,2^{n+1})})\}\Big)=\theta\Big(\{\frac1{2^k} \tau(Y_n)\}\Big)=\phi(Y),
	\end{align*}
	which proves that $\phi$ is symmetric.

\end{proof}

Next we prove the converse of Theorem \ref{thm_Pietsch_trace}, that is when a symmetric functional on $E(\cM,\tau)$ generates a $\frac12 S_+$-invariant functional on the corresponding shift-monotone sequence space $E(\Z)$. We recall, that the operator $\cD$ is defined in \eqref{eq_def_diag}.

\begin{thm}\label{thm_Pie_shift_from_trace}Assume that $E(\Z)\leftrightarrows E(\cM,\tau)$ are associated spaces and let $\phi$ be a symmetric functional on $E(\cM,\tau)$. 
	
	Then the mapping  $\theta:E(\Z)\to \C$, defined by
	\begin{equation}\label{Pie_from_trace}
		\theta(x)=\phi(\cD x), \quad x\in E(\Z)
	\end{equation}
	is a $\frac12 S_+$-invariant functional on $E(\Z)$. 
\end{thm}

\begin{proof}
	
	Let $\phi$ be a symmetric functional  on $E(\cM,\tau)$. It is clear that the mapping $\theta:E(\Z)\to \C$ defined by \eqref{Pie_from_trace} is linear.
	
	Furthermore, for any $0\leq x\in E(\Z)$ we can write
	\begin{align*}
		\theta(\frac12 S_+ x)&=\frac12\phi(\cD S_+x)= \frac12\phi\left(\sum_{n\in \Z} x_{n-1} \iota\chi_{[2^n, 2^{n+1})}\right)\\
		&=\frac12\phi\left(\sum_{n\in \Z} x_{n-1} (\iota \chi_{[2^n, \frac32 2^{n})}+\iota\chi_{[\frac32 2^{n}, 2^{n+1})})\right).
	\end{align*}
	
	Since rearrangements of the functions $\sum_{n\in \Z} x_{n-1} \chi_{[2^n, \frac32 2^{n})}$ 
	and $\sum_{n\in \Z} x_{n-1} \chi_{[\frac32 2^{n}, 2^{n+1})}$ coincide, it follows that $\mu\Big(\sum_{n\in \Z} x_{n-1} \iota \chi_{[2^n, \frac32 2^{n})}\Big)$ and $\mu\Big(\sum_{n\in \Z} x_{n-1} \iota\chi_{[\frac32 2^{n}, 2^{n+1})}\Big)$ are equal. Therefore,  we have
	$$\theta(\frac12 S_+x)=\phi\left(\sum_{n\in \Z} x_{n-1} \iota\chi_{[2^n, \frac32 2^{n})}\right).$$
	
	Similarly, since the rearrangements of the functions $\sum_{n\in \Z} x_{n-1} \chi_{[2^n, \frac32 2^{n})}$ and $\sum_{n\in \Z} x_{n-1} \chi_{[2^{n-1}, 2^{n})}$ coincide, it follows that
	$$\theta(\frac12S_+x) = \phi\left( \sum_{n\in \Z} x_{n-1} \iota \chi_{[2^{n-1}, 2^{n})}\right)=\phi(\cD x)=\theta(x).$$
	Hence, $\theta$ is an $\frac12S_+$-invariant linear functional on $E(\Z)$.
\end{proof}

\begin{rem}\label{rem_from_trace_not_shift}
	We note that, in general, the assertion of Theorem \ref{thm_Pie_shift_from_trace} does not hold if one replaces the assumption that $\phi$ is symmetric functional by the assumption that $\phi$ is a trace on $E(\cM,\tau)$. Indeed, let $\cM=L_\infty(0,\infty)$, $E(\cM,\tau)=L_1(0,\infty)$. Consider the functional  $\phi:L_1(0,\infty)\to \C$ given by 
	$$\phi(f)=\int_0^\infty \frac{ f(t)dt}{(1+t)^2}.$$
	Since the algebra $L_\infty(0,\infty)$, it trivially follows that $\phi$ is a trace. At the same time, it is clear that $\phi$ is not a symmetric functional. 
	The functional $\theta$, defined as in \eqref{Pie_from_trace}, is then given by
	$$\theta(x)=\sum_{n\in\Z}x_n \frac{2^n}{(1+2^{n+1})(1+2^n)}.$$
	One can easily check that $\theta$ is not $\frac12S_+$-invariant functional. 
\end{rem}

We now show that the correspondence given in Theorem \ref{thm_Pietsch_trace} is, in fact, a bijective correspondence between symmetric functionals on a Calkin space $E(\cM,\tau)$ and $\frac12 S_+$-invariant functionals on the associated shift-monotone sequence space $E(\Z)$. We start with auxiliary lemmas.

%
%

Recall that  for a Calkin space $E(0,\infty)$ the center $Z_E$ is defined as a linear hull of the set
$\{ f_1 - f_2 : 0\le f_1, f_2 \in E, f_1^* = f_2^* \}.$

\begin{lem}\label{inv}
	Let $E(0,\infty)$ be a Calkin space on $(0,\infty)$ and $E(\Z)$ be the corresponding shift-monotone space on $\Z$. We have that 
	$$ f- D \Phi_{\rm av}(f) \in Z_E, \ \forall \ 0 \le f \in E(0,\infty),$$
	where $\Phi_{\rm av}(f)=\{\frac1{2^n}\int_{2^n}^{2^{n+1}}f^*(s)ds\}_{n\in\Z}.$
\end{lem}

\begin{proof} 
	For every positive $f \in E(0,\infty)$, we have $f-f^* \in Z_E$. 
	Therefore, since $Z_E$ is a linear space it is sufficient to prove the statement for $f=f^*$. 
	
	It is easy to see that for every $k\in\Z$ we have
	$$\int_{2^k}^{2^{k+1}} \left(f(z)- D \Phi_{\rm av}(f)(z)\right) dz=0$$
	and so 
	$$\int_0^\infty \left(f(z)- D \Phi_{\rm av}(f)(z)\right) dz=0.$$
	
	It follows from \cite{Kwapien} that the function $f- D \Phi_{\rm av}(f)$ can be written as a difference of two functions, such that their rearrangements coincide. Hence, $f - D \Phi_{\rm av} f \in Z_E$.
	
	
\end{proof}

\begin{lem}\label{lem_suf_dyadic}
	Let $E(\cM,\tau)\subset S_0(\cM,\tau)$ be a Calkin space. For any symmetric functional $\phi$ on  $E(\cM,\tau)$ we have 
	$$\phi(A)=\phi(\cD\Phi_{\rm av} A), \quad A\geq 0,$$
	where $\Phi_{\rm av}(A):=\{\frac1{2^n}\int_{2^n}^{2^{n+1}}\mu(s,A)ds\}_{n\in\Z}$.
\end{lem}
\begin{proof}
	Let $0\leq A\in E(\cM,\tau)$ be fixed and let $\iota_A$ be the isomorphism as in Theorem~\ref{atomless_diagonal}. Since $\phi$ is a symmetric functional and both $\iota$ and $\iota_A$ preserve $\mu(\cdot)$, it follows that 
	$$\phi(\cD\Phi_{\rm av} A)=\phi(\iota D\Phi_{\rm av} A)=\phi(\iota_A D\Phi_{\rm av} A).$$
	
	By the construction on $\iota_A$ we have $\phi(A-\iota_A D\Phi_{\rm av} A)=\phi(\iota_A \mu(A)-\iota_A D\Phi_{\rm av} \mu(A)).$ By Lemma \ref{inv} we have that $\mu(A)-D\Phi_{\rm av} \mu(A)\in Z_E$. 
	Since every symmetric vanishes on the center, it follows that 
	$\phi(A-\iota_A D\Phi_{\rm av} A)=\phi(\iota_A f-\iota_A g)=0,$ as required.
\end{proof}

\begin{cor}\label{cor_one-to-one}
	Let $\cM$ be an atomless (or atomic) von Neumann algebra and let $E(\cM,\tau)$ and $E(\Z)$ be associated spaces, such that $E(\cM,\tau)\subset S_0(\cM,\tau)$. Then the rule 
	\begin{align*}
		\theta&\mapsto \phi,\\
		\phi(X)&=\theta(\{\frac1{2^k}\tau(X_k)\}_{k\in\Z}), \quad X\in E(\cM,\tau),
	\end{align*}
	where $X=\sum_{k\in\Z} X_K$ is an $E(\Z)$-dyadic representation of $X\in E(\cM,\tau)$, gives a bijective correspondence between all symmetric functionals on $E(\cM,\tau)$ and $\frac12 S_+$-invariant functionals $\theta$ on $E(\Z)$. The inverse of this rule  is given by \eqref{Pie_from_trace}. 
\end{cor}
\begin{proof}By Theorems \ref{thm_Pietsch_trace} and \ref{thm_Pie_shift_from_trace} the rule $\theta\mapsto \phi$ and its inverse are well-defined. Therefore, we only need to show that this rule is one-to-one.

	Let $\theta$ be $\frac12 S_+$-invariant functional on $E(\Z)$ and let $\phi$ be a symmetric functional on $E(\cM,\tau)$, defined by \eqref{eq_Pie_trace_from_shift}. Note that, since $\iota$ is a trace-preserving isomorphism, it follows that $\cD x=\sum_{k\in\Z} D_k,$ where the operators $D_k=x_k\iota \chi_{[2^k, 2^{k+1})}$ form a $E(\Z)$-dyadic representation of $\cD x$. Therefore, using again the fact that $\iota$ is trace preserving, we obtain 
	\begin{align*}
		\phi(\cD x)=\theta(\{\frac1{2^k}\tau(x_k \iota \chi_{[2^k, 2^{k+1})})\}_{k\in\Z})&=\theta(\{x_k\frac1{2^k} m([2^k, 2^{k+1}))\}_{k\in\Z})\\
		&=\theta(\{x_k\}_{k\in\Z}).
	\end{align*}
	Thus, $\theta\mapsto \phi$ is one-to-one. 
	
	Let $\phi$ be a symmetric functional on $E(\cM,\tau)$ and let $\theta_\phi$ be $\frac12 S_+$-invariant functional defined by \eqref{Pie_from_trace}. Denote by $\phi_1$ the symmetric functional on $E(\cM,\tau)$ defined by $\theta_\phi$ via \eqref{eq_Pie_trace_from_shift}. We claim that $\phi=\phi_1$. By Lemma  \ref{lem_suf_dyadic} it is sufficient to show that 
	$\phi(\cD\Phi_{\rm av} X)=\phi_1(\cD\Phi_{\rm av} X)$ for any $0\leq X\in E(\cM,\tau)$. 
	By definition of $\phi_1$ and $\theta_\phi$ we have \begin{align*}
		\phi_1(\cD \Phi_{\rm av} X)&=\theta_\phi\Big(\Big\{\frac1{2^n} \tau(\frac1{2^n}\int_{2^n}^{2^{n+1}}\mu(s,X)ds \cdot \iota \chi_{[2^n,2^{n+1})}\Big\}_{n\in\Z}\Big)\\
		&=\theta_\phi(\Phi_{\rm av}(X))=\phi(\cD\Phi_{\rm av} X).
	\end{align*}
	This proves the surjectivity of the correspondence $\theta\mapsto \phi$ given by \eqref{eq_Pie_trace_from_shift}. 
\end{proof}

\subsection{Classes of symmetric functionals}

We shall consider the following classes of symmetric functionals on Calkin spaces:

\begin{defn}\label{SF}Let $E(\cM,\tau)$ be a Calkin space. A symmetric functional $\phi$ on  $E(\cM,\tau)$ is called
	
	\begin{enumerate}
		\item supported at infinity if $\phi(X E^{|X|}(a,+\infty))=0$ for every $a>0$ and every $X \in E(\cM,\tau)$;
		\item supported at zero if $\phi(X E^{|X|}(0,a))=0$ for every $a>0$ and every $X \in E(\cM,\tau)$;
		\item normalised if $\phi(X)=1$ for every $X\in E(\cM,\tau)$ such that $\mu(X) = D\bf{1}$.
	\end{enumerate}
\end{defn}

Informally, symmetric functionals supported at infinity does not depend on ``large'' values of a generalised singular values function,
whereas symmetric functionals supported at zero does not depend on ``small'' values of a generalised singular values function.

We use an unusual normalisation to avoid unnecessary constants. This normalisation is natural in Pietsch-type constructions of symmetric functionals and first appeared in~\cite{P_PDO}.

Also we consider the various classes of $\frac12 S_+$-invariant functionals on shift-invariant spaces $E(\Z)$.
\begin{defn}Let $E(\Z)$ be a shift-invariant space.  A $\frac12 S_+$-invariant functional $\theta$ on $E(\Z)$ is called
	\begin{enumerate}
		\item supported at $+\infty$ if $\theta( \chi_{(-\infty,a)})=0$ for every $a\in \Z$;
		\item supported at $-\infty$ if $\theta( \chi_{(a,+\infty)})=0$ for every $a\in \Z$;
		\item normalised if $\theta(\chi_{\Z})=1$.
	\end{enumerate}
\end{defn}

\begin{rem}\label{spec_sf}
	The bijection between the set of all symmetric functionals $\phi$ on a Calkin space $E(\cM,\tau)$ and the set of all $\frac12 S_+$-invariant linear functionals $\theta$ on the corresponding shift-invariant space $E(\Z)$ described in Corollary \ref{cor_one-to-one} can be specified as follows:
	\begin{enumerate}
		\item  $\phi$ is positive if and only if $\theta$ is positive;
		\item $\phi$ is supported at infinity if and only if $\theta$ is supported at $+\infty$;
		\item $\phi$ is supported at zero if and only if $\theta$ is supported  at $-\infty$;
		\item  $\phi$ is normalised if and only if $\theta$ is normalised.
	\end{enumerate}
	
\end{rem}

\section{Pietsch correspondence and completeness} \label{sec_norms}

In this section we show that the Pietsch correspondence, given in Theorem \ref{thm_corr_nc} extends to a correspondence of symmetrically $\Delta$-normed operator space and $\Delta$-normed shift-monotone spaces. Thus, we extend \cite[Theorem 7.9]{P_trI} to the setting of an arbitrary semifinite von Neumann algebra.  Furthermore, we prove that this  correspondence preserves completeness, extending \cite{LPSZ}. Throughout this section we assume that $\cM$ is a semifinite atomless or atomic (with atoms of equal trace) von Neumann algebra equipped with a faithful normal semifinite trace $\tau$.

\subsection{Correspondence of $\Delta$-norms}
For convenience of the reader, we recall the definition of $\Delta$-norms (see e.g. \cite{KPR}). 

Let $\Omega$ be a vector space over the field $\mathbb{C}$.
A function $\|\cdot\|$ from $\Omega$ to $[0,\infty)$ is a $\Delta$-norm, if for all $x,y \in \Omega$ the following properties hold:
\begin{enumerate}
	\item $\|x\| \geqslant 0$, $\|x\| = 0 \Leftrightarrow x=0$;
	\item $\|\alpha x\| \leqslant \|x\|$ for all $|\alpha| \le1$;
	\item $\lim _{\alpha \rightarrow 0}\|\alpha x\| = 0$;
	\item $\|x+y\| \le C_\Omega \cdot (\|x\|+\|y\|)$ for a constant $C_\Omega\geq 1$ independent of $x,y$.
\end{enumerate}
The couple $(\Omega, \|\cdot\|)$ is called a $\Delta$-normed space.
It is well-known that every $\Delta$-normed space $(\Omega,\|\cdot\|)$ is a topological vector space with a metrizable topology \cite{KPR} and conversely every metrizable space can be equipped with a $\Delta$-norm (see e.g. \cite{Koethe}, \cite{KPR}).

\begin{defn}\label{opspace}
	Let $E(\cM,\tau)$ be a vector subspace of $S({\mathcal{M}, \tau})$
	equipped with a $\Delta$-norm $\|\cdot\|_{E}$. We say that
	$ E(\cM,\tau)$ is a \textit{symmetric $\Delta$-normed operator space} (on
	$\mathcal{M}$, or in $S({\mathcal{M}, \tau})$) if $X\in
	E(\cM,\tau)$ and every $Y\in S({\mathcal{M}, \tau})$ the
	assumption $\mu(Y)\leq \mu(X)$ implies that $Y\in E(\cM,\tau)$ and
	$\|Y\|_E\leq \|X\|_E$.
	In the case, when $(E(\cM,\tau), \|\cdot\|_E)$ is complete, we say that $(E(\cM,\tau), \|\cdot\|_E)$ is a complete symmetric $\Delta$-normed operator space. 
\end{defn}

Note that, as in the case of quasi-normed symmetric spaces  (see e.g. \cite{KS08b, SC} and references
therein), every symmetric $\Delta$-normed  operator space $E(\cM,\tau)$ is (an
absolutely solid) $\mathcal{M}$-bimodule of $S\left(\mathcal{M},
\tau \right)$, that is if $X\in E(\cM,\tau)$, $A,B\in \cM$, then $AXB\in E(\cM,\tau)$ and $$\|AXB\|_E\leq  \|A\|_\infty\|B\|_\infty \|X\|_E,$$ whenever $\|A\|_\infty,\|B\|_\infty\leq 1.$

\begin{rem}It is well-known that the algebra $\smt$ equipped with the measure topology is a complete metrizable topological algebra \cite{Ne}. In particular, $\smt$ can be equipped with a $\Delta$-norm, which generates the measure topology.  In the present paper whenever we consider a symmetric $\Delta$-normed operator space $E(\cM,\tau)$, we assume that it is a proper subspace of $\smt$, since completeness of $\smt$ is essential in the proof of Theorem \ref{thm_complete} below.
\end{rem}

We recall the following result \cite[Lemma 2.4]{HLS}.
\begin{lem}\label{embedding}
	If $E(\cM,\tau)\subseteq S\left(\cM, \tau \right) $ is a symmetric $\Delta$-normed space, then the embedding of $(E(\cM,\tau),\|\cdot\|_E)$ in $(S\left(\cM, \tau \right),t_\tau)$ is
	continuous, that is, if $\left\{ x_{n}\right\}
	_{n=1}^{\infty }$ is a sequence in $E(\cM,\tau)$ satisfying $\left\Vert
	x_{n}\right\Vert _{E}\rightarrow 0$, then $x_{n}\overset{t_\tau}{%
		\rightarrow }0$.
\end{lem}

Note that if $E(0,\infty)$ is a symmetric $\Delta$-normed function space, then the dilation operator $\sigma_s$ is a bounded operator on $E(0,\infty)$ for any $s>0$, and
\begin{equation}\label{norm_sigma}
	\|\sigma_{2^k}f\|_E\leq (2C_E)^k\|f\|_E,\quad k\in\N.
\end{equation}
The proof of this claim is similar to \cite[Section II.4]{KPS} and is therefore omitted.  


\begin{defn}
	Let $(E(\cM,\tau),\|\cdot\|_E)$ be a symmetrically $\Delta$-normed space on $\cM$. Following A.Pietsch \cite{P_trI} we say that the $\Delta$-norm $\|\cdot\|_E$ is \emph{stable} if $\mu(2^n,X)\leq \mu(2^n,Y),\, n\in\Z,\, X,Y\in E(\cM,\tau)$ implies that $\|X\|_E\leq \|Y\|_E$.
\end{defn}

As the following lemma shows, one can always assume that a symmetric $\Delta$-normed space $(E(\cM,\tau),\|\cdot\|_{E(\cM,\tau)})$ is equipped with a stable $\Delta$-norm $\|\cdot\|_{E(\cM,\tau)}$ (cf. \cite[Proposition 7.5]{P_trI}).
\begin{lem}\label{stable}
	For any symmetrically $\Delta$-normed space $(E(\cM,\tau), \|\cdot\|_{E(\cM,\tau)})$ there exists an equivalent stable (symmetric) $\Delta$-norm.
\end{lem}
\begin{proof}
	We define 
	$$\|X\|^\sim_{E(\cM,\tau)}:=\|\cD\{\mu(2^n,X)\}\|_{E(\cM,\tau)}=\|D\{\mu(2^n,X)\}\|_{E(0,\infty)},\ X\in E(\cM,\tau).$$
	
	Since $\mu(X)\leq D\{\mu(2^n,X)\}\leq sigma_2\mu(X)$ for any $X\in E(\cM,\tau)$, it follows  that 
	$$\|X\|_{E(\cM,\tau)}=\|\mu(X)\|_{E(0,\infty)}\leq \|D\{\mu(2^n,X)\}\|_{E(0,\infty)}\stackrel{\eqref{norm_sigma}}{\leq}2C_E\|\mu(X)\|_{E(0,\infty)}.$$
	Thus, $\|X\|_{E(cM,\tau)}\leq \|X\|^\sim_{E(\cM,\tau)}\leq 2C_E\|X\|_{E(\cM,\tau)}.$ The fact that $\|\cdot\|^\sim_{E(\cM,\tau)}$ is a $\Delta$-norm on $E(\cM,\tau)$ now follows. 
\end{proof}

Next, we introduce the notion of a $\Delta$-normed shift-monotone sequence space. 

\begin{defn}\label{def_si_normed}
	A shift-invariant space $E(\Z)\subset S(\Z)$ equipped with a $\Delta$-norm $\|\cdot\|_E$ is called a \emph{$\Delta$-normed shift-monotone space} if
	\begin{itemize}
		\item[(i).] $x\in S(Z), y\in E(\Z)$ and $o(x)\leq o(y)$ implies that $x\in E(\Z)$ and $\|x\|_E\leq \|y\|_E$.
		\item[(ii).] The shift-operator $S_+$ is bounded on $E(\Z)$.
	\end{itemize}
	If, in addition, $(E(\Z),\|\cdot\|_{E(\Z)})$ is a complete space, then $E(\Z)$ is called \emph{complete $\Delta$-normed shift-monotone space} on $\Z$. Similarly, one can define $\Delta$-normed shift-monotone spaces on $\Z_-$ and $\Z_+$.
\end{defn}

Following convention of Remark \ref{rem_diff_types} we write $E(\cM,\tau)\leftrightarrows E(\Z)$ for the associated spaces  even in the case when $\cM$ is atomless with $\tau(\mathbf{1})=1$ (so that the corresponding sequence space is indexed by $\Z_-$) or when $\cM$ is atomic (so that the corresponding sequence space is indexed by $\Z_+$). The following theorem extends \cite[Theorem 7.9]{P_trI} to the setting of semifinte von Neumann algebras. 

\begin{thm}
	\label{thm_corr_nc_norms}Let $\cM$ be an atomless or atomic von Neumann algebra equipped with a faithful normal semifinite trace $\tau$ and let $E(\cM,\tau)$ and $E(\Z)$ be associated spaces (via Theorem \ref{thm_corr_nc}). The rule
	\begin{align}\label{eq_Pie_norm_corr}
		\|\cdot\|_{E(\cM,\tau)}&\mapsto \|\cdot\|_{E(\Z)}: &\|x\|_{E(\Z)}=\|\cD x\|_{E(\cM,\tau)}, \quad x\in E(\Z),\\
		\|\cdot\|_{E(\Z)}&\mapsto \|\cdot\|_{E(\cM,\tau)}: &\|X\|_{E(\cM,\tau)}=\|\Phi X\|_{E(\Z)}, \quad X\in E(\cM,\tau),\nonumber
	\end{align}
	extends Pietsch correspondence $E(\cM,\tau)\leftrightarrows E(\Z)$ of Theorem \ref{thm_corr_nc} up to a bijective correspondence of symmetric $\Delta$-normed operator spaces $(E(\cM,\tau),\|\cdot\|_{E(\cM,\tau)})$ with stable $\Delta$-norm $\|\cdot\|_{E(\cM,\tau)}$  and $\Delta$-normed shift-invariant sequence spaces $(E(\Z),\|\cdot\|_{E(\Z)})$.

\end{thm}
\begin{proof}
	Suppose firstly that $(E(\cM,\tau),\|\cdot\|_{E(\cM,\tau)})$ is symmetric $\Delta$-normed operator space with stable $\Delta$-norm. We claim that $(E(\Z),\|\cdot\|_{E(\Z)})$ with $\|\cdot\|_{E(\Z)}$ defined by 
	$$\|x\|_{E(\Z)}=\|\cD x\|_{E(\cM,\tau)},\qquad x\in E(\Z),$$
	is $\Delta$-normed shift-invariant space. It is clear that $\|\cdot\|_{E(\Z)}$ is a $\Delta$-norm (in particular, the property (iv) is readily verified, since $\cD$ is linear).

	Assume that $x\in E(\Z)$, $y\in S(\Z)$ are such that $o(y)\leq o(x).$
	By Lemma~\ref{order_vs_rear} we have 
	$$\mu(2^n, \cD y)=\mu(2^n, Dy)=o_n(y)\leq o_n(x)=\mu(2^n, \cD x),\quad n\in\Z.$$
	Since the $\Delta$-norm $\|\cdot\|_{E(\cM,\tau)}$ is stable, it follows that  
	\begin{align*}
		\|y\|_{E(\Z)}&=\|\cD y\|_{E(\cM,\tau)}\leq \|\cD x\|_{E(\cM,\tau)}=\|x\|_{E(\Z)}.
	\end{align*}

	Furthermore, for any $x\in E(\Z)$, we have 
	$$D(S_+x)=\sum_{n\in\Z}x_{n-1}\chi_{[2^n,2^{n+1})}=\sigma_2 \Big(\sum_{n\in\Z}x_{n-1}\chi_{[2^{n-1},2^{n})}\Big)=\sigma_2(Dx).$$
	Therefore, 
	\begin{align*}
		\|S_+x\|_{ E(\Z)}&=\|\cD(S_+x)\|_{E(\cM,\tau)}=\|DS_+x\|_{E(0,\infty)}\\
		&=\|\sigma_2(Dx)\|_{E(0,\infty)}\stackrel{\eqref{norm_sigma}}{\leq} 2C_{E(\cM,\tau)}\|Dx\|_{E(0,\infty)}=2C_E\|x\|_{E(\Z)},
	\end{align*}
	that is $S_+$ is a bounded operator on $ E(\Z)$.
	Thus, $ E(\Z)$ is a $\Delta$-normed shift-monotone sequence space.

	Conversely, suppose that $(E(\Z),\|\cdot\|_{E(\Z)})$ is $\Delta$-normed shift-invariant space. It is clear that $\|\cdot\|_{E(\cM,\tau)}$ satisfies the first three axioms of a $\Delta$-norm. Suppose that $X,Y\in E(\cM,\tau)$. By \eqref{eq_Phi_shift} we have 
	\begin{align}\label{eq_mod_conc}
		\begin{split}
			\|X+Y\|_{E(\cM,\tau)}&=\|\Phi(X+Y)\|_{E(\Z)}\leq \|S_+(\Phi X+\Phi Y)\|_{E(\Z)}\\
			&\leq C_{E(\Z)}\|S_+\|(\|\Phi X\|_{E(\Z)}+\|\Phi Y\|_{E(\Z)})\\
			&=C_{E(\Z)}\|S_+\|(\|X\|_{E(\cM,\tau)}+\| Y\|_{E(\cM,\tau)}).
		\end{split}
	\end{align}
	Thus, $\|\cdot\|_{E(\cM,\tau)}$ is a $\Delta$-norm on the associated Calkin space $E(\cM,\tau)$. 
	
	Now, let $X\in S(\cM,\tau)$, $Y\in E(\cM,\tau)$ be such that $\mu(X)\leq \mu(Y)$.
	It follows from definition of $\Phi$, that $\Phi X\leq \Phi Y$, and so  
	$$\|X\|_{E(\cM,\tau)}=\|\Phi X\|_{E(\Z)}\leq\|\Phi Y\|_{E(\Z)}= \|Y\|_{E(\cM,\tau)},$$ which suffices to show that $(E(\cM,\tau),\|\cdot\|_{E(\cM,\tau)})$ is a symmetric $\Delta$-normed operator space. The stability of the $\Delta$-norm $\|\cdot\|_{E(\cM,\tau)}$ follows directly from the definition of $\Phi$.

	Next, we prove that the rule \eqref{eq_Pie_norm_corr} is bijective. 
	For any $X\in E(\cM,\tau)$, we have that $\mu(2^n,X)=\mu(2^n, \cD\Phi X)$, and therefore, stability of the $\Delta$-norm $\|\cdot\|_{E(\cM,\tau)}$ implies that 
	\begin{align*}
		\|X\|_{E(\cM,\tau)}=\|\cD\Phi X\|_{E(\cM,\tau)}.
	\end{align*}
	
	Furthermore, for any $x\in E(\Z)$ by Lemmas \ref{order_vs_rear} and \ref{exp_f_f} we have
	\begin{align*}
		\|x\|_{E(\Z)}&=\|o(x)\|_{E(\Z)}=\|\cD o(x)\|_{E(\cM,\tau)}=\|\Phi \cD o(x)\|_{E(\Z)}\\
		&=\|o(x)\|_{E(\Z)}=\|x\|_{E(\Z)},
	\end{align*}
	which show that the rule \eqref{eq_Pie_norm_corr} is a bijective correspondence.

\end{proof}

\subsection{Completeness of associated spaces}

We now prove that the Pietsch correspondence in Theorem \ref{thm_corr_nc_norms} preserves completeness. We recall the following auxiliary lemma. The proof of this lemma can be found in e.g. \cite[Criterion 4.2]{LPSZ}, \cite[Lemma 3.7]{HLS}.

\begin{lem}
	\label{crit}Let $(Z,\|\cdot\|)$ be a complete $\Delta$-normed space with constant $C_Z$. Then  the convergence of the series $\sum_{k=1}^\infty C_Z^k\|x_k\|, x_k\in Z,$ implies that the series $\sum_{k=1}^\infty x_k$ converges in $Z$.
	In this case,
	$\|\sum_{k=1}^\infty x_k\|=\sum_{k=1}^\infty C_Z^k\|x_k\|.$
\end{lem}

\begin{thm}\label{thm_complete}
	Pietsch correspondence $(E(\cM,\tau),\|\cdot\|_{E(\cM,\tau)})\leftrightarrows (E(\Z),\|\cdot\|_{E(\Z)})$ of Theorem \ref{thm_corr_nc_norms} preserves completeness. That is, if either $(E(\cM,\tau),\|\cdot\|_{E(\cM,\tau)})$ or $(E(\Z), \|\cdot\|_{E(\Z)})$ is complete, then the associated space is complete too.
\end{thm}
\begin{proof}
	Suppose firstly that $(E(\Z),\|\cdot\|_{E(\Z)})$ is a complete $\Delta$-normed shift-monotone space. 
	Let $\{X_k\}_{k=1}^\infty$ be a Cauchy sequence in $E(\cM,\tau)$. By Theorem \ref{thm_corr_nc_norms} the pair  $(E(\cM,\tau), \|\cdot\|_{E})$ is a symmetric  $\Delta$-normed  space on $\cM$. Hence, Lemma \ref{embedding} implies
	that the embedding $E(\cM,\tau)\subset S(\cM,\tau)$ is continuous, and therefore, there exists $X\in S(\cM,\tau)$ such that $X_k\to X$ with respect to the measure topology.
	
	Since $\{X_k\}_{k=1}^\infty$ is a Cauchy sequence, for every $k\in\N$ we can choose $h_k$ such that
	\begin{equation}\label{estim_Cauchy}
		\|X_h-X_{h_k}\|\leq (2C_{E(\Z)}\|S_+\|)^{-2k}\leq (2C_{E(\cM,\tau)})^{-2k}
	\end{equation} for $h\geq h_k$, where the second inequality follows from the fact that $C_{E(\cM,\tau)}\leq C_{E(\Z)}\|S_+\|$ (see \eqref{eq_mod_conc}). We set
	$$Y_k=X_{h_{k+1}}-X_{h_k}, \ k=1,2,...$$
	It follows that the sequence of partial sums of $\sum_{k=1}^\infty Y_k$ is a Cauchy sequence in $(E(\cM,\tau), \|\cdot\|_{E(\cM,\tau)})$. Hence, it is a Cauchy sequence in measure topology on $S(\cM,\tau)$,
	and therefore, the series $\sum_{k=1}^\infty Y_k$ converges in the measure topology to $X-X_{h_1}$.
	
	First we claim, that for every fixed $h\in\N$, the series $\sum_{k=h+1}^\infty \sigma_{2^{k-h}}\mu(Y_k)$ converges with respect to the measure topology on $S(0,\infty)$.
	
	For $h<m$ we have 
	\begin{align*}
		\Big\|\sum_{k=h+1}^m \sigma_{2^{k-h}}\mu(Y_k)\Big\|_{E}&\leq \sum_{k=h+1}^mC_{E}^k\|\sigma_{2^{k-h}}\mu(Y_k)\|_{E}\\
		&\stackrel{\eqref{norm_sigma}}{\leq} \sum_{k=h+1}^mC_{E}^k(2C_{E})^{k-h}\|Y_k\|_{E}\\
		&\stackrel{\eqref{estim_Cauchy}}{\leq}\sum_{k=h+1}^mC_{E}^k(2C_{E})^{k-h}(2C_{E})^{-2k}\\
		&=(2C_{E})^{-h}\sum_{k=h+1}^m2^{-k}\to (4C_{E})^{-h} \text{ as } m\to\infty.
	\end{align*}
	Hence, using again Lemma \ref{embedding}, we obtain that  for every fixed $h\in\N$ the series $\sum_{k=h+1}^\infty \sigma_{2^{k-h}}\mu(Y_k)$ converges with respect to the measure topology.
	By Lemma~\ref{un_maj} we have that
	\begin{align}\label{req_maj}
		\sum_{k=h+1}^\infty Y_k=\sum_{k=1}^\infty Y_{k+h}\triangleleft 2 \sum_{k=1}^\infty \sigma_{2^k}\mu(Y_{k+h})= 2 \sum_{k=h+1}^\infty \sigma_{2^{k-h}}\mu(Y_{k}).
	\end{align}
	
	Secondly, we claim that for every $h\in\N$ the series $\sum_{k=h+1}^\infty S_+^{k-h}\Phi(Y_k)=\sum_{k=1}^\infty S_+^{k}\Phi(Y_{k+h})$ converges in $E(\Z)$ and $\|\sum_{k=h+1}^\infty S_+^{k-h}\Phi(Y_k)\|_{E(\Z)}\to 0$ as $h\to \infty$.
	We have
	\begin{align*}
		\|S_+^{k}\Phi(Y_{k+h})\|_{E(\Z)}&\leq \|S_+\|^{k}\|\Phi(Y_{k+h})\|_{E(\Z)}= \|S_+\|^{k}\|Y_{k+h}\|_{E(\cM,\tau)}\\
		&\stackrel{\eqref{estim_Cauchy}}{\leq} \frac{\|S_+\|^{k}}{(2C_{E}\|S_+\|)^{2k+2h}}\leq(2C_{E})^{-2k-2h}.
	\end{align*}
	Hence,
	$$\sum_{k=1}^\infty C_E^k\|S_+^{k}\Phi(Y_{k+h})\|_E\leq  (2C_E)^{-2h}\sum_{k=1}^\infty (4C_E)^{-k},$$
	and therefore, by Lemma \ref{crit} the series
	\begin{align}\label{def_a_h}
		a^{(h)}:=\sum_{k=h+1}^\infty S_+^{k-h}\Phi(Y_k)
	\end{align} converges in $E(\Z)$ and
	\begin{align}\label{norm_a_h}
		\|a^{(h)}\|_{E(\Z)}=\Big\|\sum_{k=h+1}^\infty S_+^{k-h}\Phi(Y_k)\Big\|_{E(\Z)}\leq (2C_E)^{-2h}\sum_{k=1}^\infty (4C_E)^{-k}\to 0
	\end{align}
	as $h\to \infty$.
	
	Now, using \eqref{req_maj} and \eqref{def_a_h} for every $n\in\Z$ we have
	\begin{align*}
		\Big(\Phi \Big(\sum_{k=h+1}^\infty Y_k\Big)\Big)_{n+1}&=\mu\Big(2^{n+1},\sum_{k=h+1}^\infty Y_k\Big)\leq 
		\frac1{2^{n}}\int_{2^{n}}^{2^{n+1}}\mu(s,\sum_{k=h+1}^\infty Y_k)ds\\
		&\stackrel{\eqref{req_maj}}{\leq} \frac1{2^{n-1}}\int_{2^{n-1}}^{2^{n+1}}\sum_{k=h+1}^\infty
		\sigma_{2^{k-h}}\mu(s,Y_k)ds
		\\
		&
		=\sum_{k=h+1}^\infty \frac1{2^{n-k+h-1}}\int_{2^{n-k+h-1}}^{2^{n-k+h+1}}\mu(s,Y_k)ds\\
		&=\sum_{k=h+1}^\infty \frac1{2^{n-k+h-1}}\int_{2^{n-k+h-1}}^{2^{n-k+h}}\mu(s,Y_k)ds\\
		&+\sum_{k=h+1}^\infty \frac12\frac1{2^{n-k+h}}\int_{2^{n-k+h}}^{2^{n-k+h+1}}\mu(s,Y_k)ds\\
		&\leq\sum_{k=h+1}^\infty \big(\Phi Y_k\big)_{n-k+h-1}+\sum_{k=h+1}^\infty \frac12\big(\Phi Y_k\big)_{n-k+h}\\
		&=S_+\sum_{k=h+1}^\infty S_+^{k-h}\big(\Phi Y_k\big)_n+\frac12\sum_{k=h+1}^\infty S_+^{k-h}\big(\Phi Y_k\big)_n\\
		&\stackrel{\eqref{def_a_h}}{=}(S_+a^{(h)})_n+\frac12 (a^{(h)})_n.
	\end{align*}
	
	Thus, by the definition of the $\Delta$-norm of $E(\cM,\tau)$ and \eqref{norm_a_h} we have that
	$$\Big\|\sum_{k=h+1}^\infty Y_k\Big\|_{E(\cM,\tau)}=\Big\|\Phi\Big( \sum_{k=h+1}^\infty Y_k\Big)\Big\|_{E(\Z)}\leq (\|S_+\|+\frac12)\|a^{(h)}\|_{E(\Z)}\to  0$$
	as $h\to\infty$. That is, the series $\sum_{k=1}^\infty Y_k$ converges in $(E(\cM,\tau),\|\cdot\|_{E})$.
	Since $\sum_{k=h+1}^\infty Y_k=X-X_{m_{h+1}}$, it follows that the subsequence $\{X_{m_{h}}\}_{h=1}^\infty$ converges to $X$ in $E(\cM,\tau)$. Since a Cauchy sequence $\{X_n\}_{n=1}^\infty$ in $E(\cM,\tau)$ has a
	convergent subsequence, it converges in $E(\cM,\tau)$. Thus, $(E(\cM,\tau),\|\cdot\|_{E})$ is complete.

	Conversely, assume that $ E(\cM,\tau)$ is a complete symmetric $\Delta$-normed operator space on $\cM$. It follows from the definition of $E(0,\infty)$ that the space $(E(0,\infty),\|\cdot\|_E)$ is a complete symmetric $\Delta$-normed function space.
	Let $\{x^{(n)}\}_{n=1}^\infty$ be a Cauchy sequence in $E(\Z).$ Therefore, the sequence $\{Dx^{(n)}\}_{n=1}^\infty$ is a Cauchy sequence in $E(0,\infty)$. Since $E(0,\infty)$ is complete, there exists $f\in E(0,\infty)$ such that $Dx^{(n)}\to f$ in
	$E(0,\infty)$. By Lemma \ref{embedding}, $Dx^{(n)}\to f$ with respect to the measure topology on $S(0,\infty)$, in particular, $x_k^{(n)}=Dx^{(n)}\chi_{[2^k,2^{k+1})}\to f\chi_{[2^k,2^{k+1})}$ for every fixed $k\in \Z$.
	Setting $x_k=\lim_{n\to\infty}x_k^{(n)}$, we obtain that $f=\sum_{k\in\Z}x_k\chi_{[2^k,2^{k+1})}.$ Hence, we have that $x=\{x_k\}_{k=1}^\infty\in E(\Z)$ (since $Dx=f\in E(0,\infty)$) and
	$$\|x^{(n)}-x\|_{E(\Z)}=\|Dx^{(n)}-Dx\|_{E(\Z)}=\|Dx^{(n)}-f\|_{E(\Z)}\to 0.$$
	Thus $(E(\Z),\|\cdot\|_{E})$ is a complete space.
\end{proof}



\end{document}